\documentclass[11pt,a4paper]{amsart}

\usepackage{amsmath}
\usepackage{mathtools}
\usepackage{amsthm}
\usepackage{amssymb}
\usepackage{hyperref}

\newcommand*\fullref[3][\relax]{%
  \ifdefined\hyperref%
    {\hyperref[#3]{#2\penalty 200\ \ref*{#3}#1}}%
  \else%
    {#2\penalty 200\ \relax\ref{#3}#1}%
  \fi%
}

\usepackage{tikz}
\usetikzlibrary{calc}
\usetikzlibrary{math}

\newcommand*{\defterm}[1]{\emph{#1}}

\makeatletter
\newcommand\chyph{\penalty\@M-\hskip\z@skip}
\makeatother

\theoremstyle{definition}

\theoremstyle{plain}

\numberwithin{equation}{section}
\DeclarePairedDelimiter{\abs}{\lvert}{\rvert}

\DeclarePairedDelimiter{\parens}{\lparen}{\rparen}

\DeclarePairedDelimiterX{\gset}[2]{\{}{\}}{\,#1:#2\,}

\newcommand*{\biggg}{\bBigg@{4}}

\newcommand*{\Biggg}{\bBigg@{5}}

\makeatletter
\newcommand*{\sizeddelimiter}[2]{\bBigg@{#1}#2}
\makeatother

\newcommand*{\nset}{\mathbb{N}}

\newcommand*{\rset}{\mathbb{R}}

\DeclarePairedDelimiterX{\pres}[2]{\langle}{\rangle}{#1\,\delimsize\vert\,\mathopen{}#2}

\newtheorem*{theorem*}{Theorem}

\usepackage[
style=english,
english=british,
]{csquotes}

\DeclareMathOperator{\Inv}{Inv}

\begin{document}

\title{Visual thinking and simplicity in proof}

\author{Alan J. Cain}
\address{%
Centro de Matem\'{a}tica e Aplica\c{c}\~{o}es\\
Faculdade de Ci\^{e}ncias e Tecnologia\\
Universidade Nova de Lisboa\\
2829--516 Caparica\\
Portugal
}
\email{%
a.cain@fct.unl.pt
}
\thanks{The author was supported by the Funda\c{c}\~{a}o para a Ci\^{e}ncia e a Tecnologia (Portuguese Foundation for Science
and Technology) through an Investigador {\sc FCT} fellowship ({\sc IF}/01622/2013/{\sc CP}1161/{\sc
    CT}0001). This work was partially supported by by the Funda\c{c}\~{a}o para
a Ci\^{e}ncia e a Tecnologia (Portuguese Foundation for Science and Technology) through the project {\sc UID}/{\sc
  MAT}/00297/2013 (Centro de Matem\'{a}tica e Aplica\c{c}\~{o}es) and the project {\scshape PTDC}/{\scshape
      MHC-FIL}/2583/2014.}

\begin{abstract}
  This paper studies how spatial thinking interacts with simplicity in [informal] proof, by analysing a set of example
  proofs mainly concerned with Ferrers diagrams (visual representations of partitions of integers, and comparing them to
  proofs that do not use spatial thinking. The analysis shows that using diagrams and spatial thinking can contribute to
  simplicity by (for example) avoiding technical calculations, division into cases, and induction, and creating a more
  surveyable and explanatory proof (both of which are connected to simplicity). In response to one part of Hilbert's
  24th Problem, the area between two proofs is explored in one example, showing that between a proof that uses spatial
  reasoning and one that does not, there is a proof that is less simple than either.
\end{abstract}

\maketitle

\section{Introduction}

The use of visual thinking in mathematics --- broadly encompassing visual perception and visual imagination --- has
started to receive substantial attention from philosophers of mathematics. (See, for example,
\cite{mancosu_visualization,giaquinto_visualizing,giardino_intuition}.) This is unsurprising, given the importance
that mathematicians themselves place on visual thinking: famously, Jacques Hadamard wrote of those he surveyed:
\textquote[{\cite[p.~84]{hadamard_psychology}}]{Practically all of them avoid not only the use of mental words but
  also, just as I do, the mental use of algebraic or any other precise signs; also as in my case, they use vague
  images.}  (Notable exceptions are G.H. Hardy, who wrote:
\textquote[{\cite[p.~114]{hardy_reviewhadamard}}]{I think almost entirely in words and formulae\ldots\ I find
  thought almost impossible if my hands are cold and I cannot write in comfort}, and Bertrand Russell, who noted, in
the context of William James's distinction between the visual mind and auditory mind:
\textquote[{\cite[p.~57]{spadoni_russell}}]{I never think except in words which I imagine spoken}.)

Within the broad area of visual thinking is spatial thinking, meaning specifically thinking that involves visual
images of objects in particular geometric configurations.  This is to be distinguished from the visual thinking
involved in symbol manipulation (see, for example, \cite[ch.~10]{giaquinto_visual}) or in visualizing structure such as
ordered sets, where the geometric configuration can be varied without affecting the reasoning (see
\cite{giaquinto_cognition}).

The epistemological role of spatial thinking has been a particular focus of philosophers' attention. How can a
particular diagram (which may only illustrate a special case) yield knowledge about a general case?  (For a discussion,
see \cite{dove_canpictures}.) If a proof requires spatial thinking, is one justified in accepting the theorem? Spatial
thinking can be misleading, particularly (but not only) with regard to limits of infinite processes (see, for example,
\cite[pp.~3ff.]{giaquinto_visual}).

This paper leaves aside the epistemological controversy and instead focusses on how spatial thinking interacts with
simplicity in proof. The motivation includes two aspects of Hilbert's 24th Problem (see \cite{thiele_hilbert24}):
\textquote{criteria of simplicity} and \textquote{the area lying between the two routes}. Does spatial thinking
contribute to simplicity? Many mathematicians seem to think so; witness the famous \enquote{proofs without words}
features in \textit{Mathematics Magazine} and \textit{The College Mathematics Journal} and collected in
\cite{nelsen_proofswithoutwords1,nelsen_proofswithoutwords2}. Gardner said of these that
\textquote[{\cite[ch.~16]{gardner_mathematicalgames11}}]{in many cases a dull proof can be supplemented by a geometric
  analogue so \emph{simple} and beautiful that the truth of a theorem is almost seen at a single glance} (emphasis
added). If so, in what ways does spatial thinking lead to gains in simplicity? And what can be said about the
\textquote{area lying between} two proofs that use and that do not use spatial thinking?

The approach is to examine carefully particular proofs that use spatial thinking (for brevity, henceforth
\enquote{spatial proofs}) to see how the spatial thinking contributes to simplicity. All the proofs considered in this
paper are informal proofs. Although diagrammatic formal systems have been developed (see, for example,
\cite{miller_diagrammatic}), most proofs that employ spatial thinking are informal. Furthermore, simplicity is a
property that is more readily identified in informal proofs; as Iemhoff has noted, it seems difficult to determine the
extent to which simplicity is preserved under formalization \cite[p.~147]{iemhoff_remarks}. The main (but not sole)
focus is what can broadly be called \enquote{pebble diagrams}: representing natural numbers using dots, or, as the
Pythagoreans did, pebbles. (See \cite[pp.~100ff.]{burnet_earlygreekphilosophy} or \cite[p.~33]{burkert_lore} for
historical background and sources.) Such diagrams are still used in modern mathematics, for instance in the guise of
Ferrers diagrams of partitions (\fullref{Section}{sec:ferrers} below). There are two reasons for this choice. First,
with one exception, the \emph{theorems} discussed do not require any spatial thinking to understand. Second, considering
only proofs about discrete quantities may at least mitigate some of the concern about whether the proofs are legitimate,
since these avoid situations where limits of infinite processes play a part. For comparison, proofs that do not use
spatial thinking (for brevity, \enquote{non-spatial proofs}) are sometimes outlined insofar as necessary to compare them
with the spatial proofs. (It is perhaps worth noting that Sylvester, author of the paper in which Ferrers diagrams were
first published, declared himself perfectly satisfied by proofs that use them; he describes such a proof as
\textquote[{\cite[p.~597]{sylvester_impromptu}}]{so simple and instructive, that I am sure every logician will be
  delighted to meet with it here or elsewhere}. Ferrers diagrams were suggested by Sylvester's correspondent Ferrers;
see \cite{kimberling_originferrers} for a discussion of the origin of the name.)

Detlefsen \cite[p.~87]{detlefsen_philmath20thcentury} distinguished \emph{verificational} simplicity and
\emph{inventional} simplicity of a proof: the former is the simplicity found by a reader who sets out to work through
and check the proof; the latter is the simplicity involved in discovery of the proof in the first place. To this, one
can add \emph{explanatorial} simplicity: that is, the simplicity found by the reader who sets out to understand how the
proof explains the theorem. The notion of explanation in mathematics, and how it compares with explanation in the
natural sciences, is debated (see, for example, \cite{mancosu_mathematicalexplanation,zelcer_against}). In particular,
there is disagreement over whether proofs that use induction are explanatory \cite{hoeltje_explanation,lange_whyproofs}.

The examination of proofs here is concerned with verificational and explanatorial simplicity. Prima facie, from this
perspective, a proof has greater simplicity if it avoids technical calculations, complicated notation, division into
many cases, the use of technical results, \textquote[{\cite[pp.239--240]{tucker_rethinkingrigor}}]{deus-ex-machina
  auxiliary functions}, or (given the debate over explanatoriness mentioned in the last paragraph) induction. Length has
also been suggested as a criterion of simplicity; for instance by Thiele \& Wos \cite{thiele_hilbert24wos} and Arara
\cite{arana_onthealleged} (but see the countervailing views of of Iemhoff \cite{iemhoff_remarks}), but consideration of
length if avoided in what follows, because the difficulties of comparing lengths of proofs outside of the same formal
system are multiplied when one considers comparing proof steps that use diagrams with those that do not.


\section{Diagrams and spatial thinking}

Before proceeding, it is important to distinguish between spatial thinking and using diagrams. On the one hand, a proof
can use diagrams and not call upon spatial thinking. Standard examples would be \enquote{diagram chasing} proofs of
results using commutative diagrams (see, for instance, \cite[Lemma~VIII.4.1, pp.~202--203]{maclane_categories}). In such
proofs, an element is \enquote{chased} around a diagram of arrows between objects. These arrows represent morphisms, and
alternative paths of arrows between different objects correspond to equal compositions of homomorphisms. But at no point
is spatial thinking actually used: the diagram could be redrawn in a messy form (say, with arrows unnecessarily
criss-crossing and objects scattered arbitrarily), or even be replaced with a list of equalities of compositions of
homomorphisms and the proof could remain identical. Thus, while analyzing how diagrams of arrows can simplify proofs
(which is the very starting point of category theory, according to Mac~Lane \cite[p.~1]{maclane_categories}), such
analysis is beyond the scope of this paper.

On the other hand, a proof can involve spatial thinking and yet involve no actual diagram. An example is Gale's proof of
the Hex Theorem, which shows that a game of Hex cannot end in a draw. A Hex board is usually a rhomboidal grid of
hexagons (traditionally $11 \times 11$, but more generally $n \times n$ or even $n \times m$), with sides of the rhombus
alternately coloured black and white, like the following:
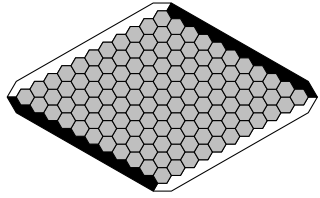
\[
%
\begin{tikzpicture}[scale=.4,rotate=30]
  \draw[fill=black,line join=bevel] ($ ($ ({((2*0)-0)*cos(30)*3mm},{(1+sin(30))*0*3mm}) + (-150:3mm) $) + (210:3mm) $) -- ($ ({((2*0)-0)*cos(30)*3mm},{(1+sin(30))*0*3mm}) + (210:3mm) $) -- ($ ({((2*0)-0)*cos(30)*3mm},{(1+sin(30))*0*3mm}) + (150:3mm) $) -- ($ ({((2*0)-1)*cos(30)*3mm},{(1+sin(30))*1*3mm}) + (210:3mm) $) -- ($ ({((2*0)-1)*cos(30)*3mm},{(1+sin(30))*1*3mm}) + (150:3mm) $) -- ($ ({((2*0)-2)*cos(30)*3mm},{(1+sin(30))*2*3mm}) + (210:3mm) $) -- ($ ({((2*0)-2)*cos(30)*3mm},{(1+sin(30))*2*3mm}) + (150:3mm) $) -- ($ ({((2*0)-3)*cos(30)*3mm},{(1+sin(30))*3*3mm}) + (210:3mm) $) -- ($ ({((2*0)-3)*cos(30)*3mm},{(1+sin(30))*3*3mm}) + (150:3mm) $) -- ($ ({((2*0)-4)*cos(30)*3mm},{(1+sin(30))*4*3mm}) + (210:3mm) $) -- ($ ({((2*0)-4)*cos(30)*3mm},{(1+sin(30))*4*3mm}) + (150:3mm) $) -- ($ ({((2*0)-5)*cos(30)*3mm},{(1+sin(30))*5*3mm}) + (210:3mm) $) -- ($ ({((2*0)-5)*cos(30)*3mm},{(1+sin(30))*5*3mm}) + (150:3mm) $) -- ($ ({((2*0)-6)*cos(30)*3mm},{(1+sin(30))*6*3mm}) + (210:3mm) $) -- ($ ({((2*0)-6)*cos(30)*3mm},{(1+sin(30))*6*3mm}) + (150:3mm) $) -- ($ ({((2*0)-7)*cos(30)*3mm},{(1+sin(30))*7*3mm}) + (210:3mm) $) -- ($ ({((2*0)-7)*cos(30)*3mm},{(1+sin(30))*7*3mm}) + (150:3mm) $) -- ($ ({((2*0)-8)*cos(30)*3mm},{(1+sin(30))*8*3mm}) + (210:3mm) $) -- ($ ({((2*0)-8)*cos(30)*3mm},{(1+sin(30))*8*3mm}) + (150:3mm) $) -- ($ ({((2*0)-9)*cos(30)*3mm},{(1+sin(30))*9*3mm}) + (210:3mm) $) -- ($ ({((2*0)-9)*cos(30)*3mm},{(1+sin(30))*9*3mm}) + (150:3mm) $) -- ($ ({((2*0)-10)*cos(30)*3mm},{(1+sin(30))*10*3mm}) + (210:3mm) $) -- ($ ({((2*0)-10)*cos(30)*3mm},{(1+sin(30))*10*3mm}) + (150:3mm) $) -- ($ ($ ({((2*0)-10)*cos(30)*3mm},{(1+sin(30))*10*3mm}) + (150:3mm) $) + (150:3mm) $) -- ($ ($ ({((2*0)-10)*cos(30)*3mm},{(1+sin(30))*10*3mm}) + (-150:3mm) $) + (210:3mm) $) -- cycle;
  \draw[fill=black,line join=bevel] ($ ($ ({((2*10)-0)*cos(30)*3mm},{(1+sin(30))*0*3mm}) + (-30:3mm) $) + (330:3mm) $) -- ($ ({((2*10)-0)*cos(30)*3mm},{(1+sin(30))*0*3mm}) + (330:3mm) $) -- ($ ({((2*10)-0)*cos(30)*3mm},{(1+sin(30))*0*3mm}) + (30:3mm) $) -- ($ ({((2*10)-1)*cos(30)*3mm},{(1+sin(30))*1*3mm}) + (330:3mm) $) -- ($ ({((2*10)-1)*cos(30)*3mm},{(1+sin(30))*1*3mm}) + (30:3mm) $) -- ($ ({((2*10)-2)*cos(30)*3mm},{(1+sin(30))*2*3mm}) + (330:3mm) $) -- ($ ({((2*10)-2)*cos(30)*3mm},{(1+sin(30))*2*3mm}) + (30:3mm) $) -- ($ ({((2*10)-3)*cos(30)*3mm},{(1+sin(30))*3*3mm}) + (330:3mm) $) -- ($ ({((2*10)-3)*cos(30)*3mm},{(1+sin(30))*3*3mm}) + (30:3mm) $) -- ($ ({((2*10)-4)*cos(30)*3mm},{(1+sin(30))*4*3mm}) + (330:3mm) $) -- ($ ({((2*10)-4)*cos(30)*3mm},{(1+sin(30))*4*3mm}) + (30:3mm) $) -- ($ ({((2*10)-5)*cos(30)*3mm},{(1+sin(30))*5*3mm}) + (330:3mm) $) -- ($ ({((2*10)-5)*cos(30)*3mm},{(1+sin(30))*5*3mm}) + (30:3mm) $) -- ($ ({((2*10)-6)*cos(30)*3mm},{(1+sin(30))*6*3mm}) + (330:3mm) $) -- ($ ({((2*10)-6)*cos(30)*3mm},{(1+sin(30))*6*3mm}) + (30:3mm) $) -- ($ ({((2*10)-7)*cos(30)*3mm},{(1+sin(30))*7*3mm}) + (330:3mm) $) -- ($ ({((2*10)-7)*cos(30)*3mm},{(1+sin(30))*7*3mm}) + (30:3mm) $) -- ($ ({((2*10)-8)*cos(30)*3mm},{(1+sin(30))*8*3mm}) + (330:3mm) $) -- ($ ({((2*10)-8)*cos(30)*3mm},{(1+sin(30))*8*3mm}) + (30:3mm) $) -- ($ ({((2*10)-9)*cos(30)*3mm},{(1+sin(30))*9*3mm}) + (330:3mm) $) -- ($ ({((2*10)-9)*cos(30)*3mm},{(1+sin(30))*9*3mm}) + (30:3mm) $) -- ($ ({((2*10)-10)*cos(30)*3mm},{(1+sin(30))*10*3mm}) + (330:3mm) $) -- ($ ({((2*10)-10)*cos(30)*3mm},{(1+sin(30))*10*3mm}) + (30:3mm) $) -- ($ ($ ({((2*10)-10)*cos(30)*3mm},{(1+sin(30))*10*3mm}) + (30:3mm) $) + (30:3mm) $) -- ($ ($ ({((2*10)-0)*cos(30)*3mm},{(1+sin(30))*0*3mm}) + (30:3mm) $) + (30:3mm) $) -- cycle;
  \draw[line join=bevel] ($ ($ ({((2*0)-0)*cos(30)*3mm},{(1+sin(30))*0*3mm}) + (-90:3mm) $) + (270:3mm) $) -- ($ ($ ({((2*0)-0)*cos(30)*3mm},{(1+sin(30))*0*3mm}) + (-150:3mm) $) + (210:3mm) $) -- ($ ({((2*0)-0)*cos(30)*3mm},{(1+sin(30))*0*3mm}) + (210:3mm) $) -- ($ ({((2*0)-0)*cos(30)*3mm},{(1+sin(30))*0*3mm}) + (270:3mm) $) -- ($ ({((2*0)-0)*cos(30)*3mm},{(1+sin(30))*0*3mm}) + (330:3mm) $) -- ($ ({((2*1)-0)*cos(30)*3mm},{(1+sin(30))*0*3mm}) + (270:3mm) $) -- ($ ({((2*1)-0)*cos(30)*3mm},{(1+sin(30))*0*3mm}) + (330:3mm) $) -- ($ ({((2*2)-0)*cos(30)*3mm},{(1+sin(30))*0*3mm}) + (270:3mm) $) -- ($ ({((2*2)-0)*cos(30)*3mm},{(1+sin(30))*0*3mm}) + (330:3mm) $) -- ($ ({((2*3)-0)*cos(30)*3mm},{(1+sin(30))*0*3mm}) + (270:3mm) $) -- ($ ({((2*3)-0)*cos(30)*3mm},{(1+sin(30))*0*3mm}) + (330:3mm) $) -- ($ ({((2*4)-0)*cos(30)*3mm},{(1+sin(30))*0*3mm}) + (270:3mm) $) -- ($ ({((2*4)-0)*cos(30)*3mm},{(1+sin(30))*0*3mm}) + (330:3mm) $) -- ($ ({((2*5)-0)*cos(30)*3mm},{(1+sin(30))*0*3mm}) + (270:3mm) $) -- ($ ({((2*5)-0)*cos(30)*3mm},{(1+sin(30))*0*3mm}) + (330:3mm) $) -- ($ ({((2*6)-0)*cos(30)*3mm},{(1+sin(30))*0*3mm}) + (270:3mm) $) -- ($ ({((2*6)-0)*cos(30)*3mm},{(1+sin(30))*0*3mm}) + (330:3mm) $) -- ($ ({((2*7)-0)*cos(30)*3mm},{(1+sin(30))*0*3mm}) + (270:3mm) $) -- ($ ({((2*7)-0)*cos(30)*3mm},{(1+sin(30))*0*3mm}) + (330:3mm) $) -- ($ ({((2*8)-0)*cos(30)*3mm},{(1+sin(30))*0*3mm}) + (270:3mm) $) -- ($ ({((2*8)-0)*cos(30)*3mm},{(1+sin(30))*0*3mm}) + (330:3mm) $) -- ($ ({((2*9)-0)*cos(30)*3mm},{(1+sin(30))*0*3mm}) + (270:3mm) $) -- ($ ({((2*9)-0)*cos(30)*3mm},{(1+sin(30))*0*3mm}) + (330:3mm) $) -- ($ ({((2*10)-0)*cos(30)*3mm},{(1+sin(30))*0*3mm}) + (270:3mm) $) -- ($ ({((2*10)-0)*cos(30)*3mm},{(1+sin(30))*0*3mm}) + (330:3mm) $) -- ($ ($ ({((2*10)-0)*cos(30)*3mm},{(1+sin(30))*0*3mm}) + (-30:3mm) $) + (330:3mm) $) -- ($ ($ ({((2*10)-0)*cos(30)*3mm},{(1+sin(30))*0*3mm}) + (-90:3mm) $) + (270:3mm) $) -- cycle;
  \draw[line join=bevel] ($ ($ ({((2*0)-10)*cos(30)*3mm},{(1+sin(30))*10*3mm}) + (90:3mm) $) + (90:3mm) $) -- ($ ($ ({((2*0)-10)*cos(30)*3mm},{(1+sin(30))*10*3mm}) + (150:3mm) $) + (150:3mm) $) -- ($ ({((2*0)-10)*cos(30)*3mm},{(1+sin(30))*10*3mm}) + (150:3mm) $) -- ($ ({((2*0)-10)*cos(30)*3mm},{(1+sin(30))*10*3mm}) + (90:3mm) $) -- ($ ({((2*0)-10)*cos(30)*3mm},{(1+sin(30))*10*3mm}) + (30:3mm) $) -- ($ ({((2*1)-10)*cos(30)*3mm},{(1+sin(30))*10*3mm}) + (90:3mm) $) -- ($ ({((2*1)-10)*cos(30)*3mm},{(1+sin(30))*10*3mm}) + (30:3mm) $) -- ($ ({((2*2)-10)*cos(30)*3mm},{(1+sin(30))*10*3mm}) + (90:3mm) $) -- ($ ({((2*2)-10)*cos(30)*3mm},{(1+sin(30))*10*3mm}) + (30:3mm) $) -- ($ ({((2*3)-10)*cos(30)*3mm},{(1+sin(30))*10*3mm}) + (90:3mm) $) -- ($ ({((2*3)-10)*cos(30)*3mm},{(1+sin(30))*10*3mm}) + (30:3mm) $) -- ($ ({((2*4)-10)*cos(30)*3mm},{(1+sin(30))*10*3mm}) + (90:3mm) $) -- ($ ({((2*4)-10)*cos(30)*3mm},{(1+sin(30))*10*3mm}) + (30:3mm) $) -- ($ ({((2*5)-10)*cos(30)*3mm},{(1+sin(30))*10*3mm}) + (90:3mm) $) -- ($ ({((2*5)-10)*cos(30)*3mm},{(1+sin(30))*10*3mm}) + (30:3mm) $) -- ($ ({((2*6)-10)*cos(30)*3mm},{(1+sin(30))*10*3mm}) + (90:3mm) $) -- ($ ({((2*6)-10)*cos(30)*3mm},{(1+sin(30))*10*3mm}) + (30:3mm) $) -- ($ ({((2*7)-10)*cos(30)*3mm},{(1+sin(30))*10*3mm}) + (90:3mm) $) -- ($ ({((2*7)-10)*cos(30)*3mm},{(1+sin(30))*10*3mm}) + (30:3mm) $) -- ($ ({((2*8)-10)*cos(30)*3mm},{(1+sin(30))*10*3mm}) + (90:3mm) $) -- ($ ({((2*8)-10)*cos(30)*3mm},{(1+sin(30))*10*3mm}) + (30:3mm) $) -- ($ ({((2*9)-10)*cos(30)*3mm},{(1+sin(30))*10*3mm}) + (90:3mm) $) -- ($ ({((2*9)-10)*cos(30)*3mm},{(1+sin(30))*10*3mm}) + (30:3mm) $) -- ($ ({((2*10)-10)*cos(30)*3mm},{(1+sin(30))*10*3mm}) + (90:3mm) $) -- ($ ({((2*10)-10)*cos(30)*3mm},{(1+sin(30))*10*3mm}) + (30:3mm) $) -- ($ ($ ({((2*10)-10)*cos(30)*3mm},{(1+sin(30))*10*3mm}) + (30:3mm) $) + (30:3mm) $) -- ($ ($ ({((2*10)-10)*cos(30)*3mm},{(1+sin(30))*10*3mm}) + (90:3mm) $) + (90:3mm) $) -- cycle;
  \draw[fill=lightgray] ({((2*0)-0)*cos(30)*3mm},{(1+sin(30))*0*3mm}) +(30:3mm) -- +(90:3mm) -- +(150:3mm) -- +(-150:3mm) -- +(-90:3mm) -- +(-30:3mm) -- cycle;
  \draw[fill=lightgray] ({((2*1)-0)*cos(30)*3mm},{(1+sin(30))*0*3mm}) +(30:3mm) -- +(90:3mm) -- +(150:3mm) -- +(-150:3mm) -- +(-90:3mm) -- +(-30:3mm) -- cycle;
  \draw[fill=lightgray] ({((2*2)-0)*cos(30)*3mm},{(1+sin(30))*0*3mm}) +(30:3mm) -- +(90:3mm) -- +(150:3mm) -- +(-150:3mm) -- +(-90:3mm) -- +(-30:3mm) -- cycle;
  \draw[fill=lightgray] ({((2*3)-0)*cos(30)*3mm},{(1+sin(30))*0*3mm}) +(30:3mm) -- +(90:3mm) -- +(150:3mm) -- +(-150:3mm) -- +(-90:3mm) -- +(-30:3mm) -- cycle;
  \draw[fill=lightgray] ({((2*4)-0)*cos(30)*3mm},{(1+sin(30))*0*3mm}) +(30:3mm) -- +(90:3mm) -- +(150:3mm) -- +(-150:3mm) -- +(-90:3mm) -- +(-30:3mm) -- cycle;
  \draw[fill=lightgray] ({((2*5)-0)*cos(30)*3mm},{(1+sin(30))*0*3mm}) +(30:3mm) -- +(90:3mm) -- +(150:3mm) -- +(-150:3mm) -- +(-90:3mm) -- +(-30:3mm) -- cycle;
  \draw[fill=lightgray] ({((2*6)-0)*cos(30)*3mm},{(1+sin(30))*0*3mm}) +(30:3mm) -- +(90:3mm) -- +(150:3mm) -- +(-150:3mm) -- +(-90:3mm) -- +(-30:3mm) -- cycle;
  \draw[fill=lightgray] ({((2*7)-0)*cos(30)*3mm},{(1+sin(30))*0*3mm}) +(30:3mm) -- +(90:3mm) -- +(150:3mm) -- +(-150:3mm) -- +(-90:3mm) -- +(-30:3mm) -- cycle;
  \draw[fill=lightgray] ({((2*8)-0)*cos(30)*3mm},{(1+sin(30))*0*3mm}) +(30:3mm) -- +(90:3mm) -- +(150:3mm) -- +(-150:3mm) -- +(-90:3mm) -- +(-30:3mm) -- cycle;
  \draw[fill=lightgray] ({((2*9)-0)*cos(30)*3mm},{(1+sin(30))*0*3mm}) +(30:3mm) -- +(90:3mm) -- +(150:3mm) -- +(-150:3mm) -- +(-90:3mm) -- +(-30:3mm) -- cycle;
  \draw[fill=lightgray] ({((2*10)-0)*cos(30)*3mm},{(1+sin(30))*0*3mm}) +(30:3mm) -- +(90:3mm) -- +(150:3mm) -- +(-150:3mm) -- +(-90:3mm) -- +(-30:3mm) -- cycle;
  \draw[fill=lightgray] ({((2*0)-1)*cos(30)*3mm},{(1+sin(30))*1*3mm}) +(30:3mm) -- +(90:3mm) -- +(150:3mm) -- +(-150:3mm) -- +(-90:3mm) -- +(-30:3mm) -- cycle;
  \draw[fill=lightgray] ({((2*1)-1)*cos(30)*3mm},{(1+sin(30))*1*3mm}) +(30:3mm) -- +(90:3mm) -- +(150:3mm) -- +(-150:3mm) -- +(-90:3mm) -- +(-30:3mm) -- cycle;
  \draw[fill=lightgray] ({((2*2)-1)*cos(30)*3mm},{(1+sin(30))*1*3mm}) +(30:3mm) -- +(90:3mm) -- +(150:3mm) -- +(-150:3mm) -- +(-90:3mm) -- +(-30:3mm) -- cycle;
  \draw[fill=lightgray] ({((2*3)-1)*cos(30)*3mm},{(1+sin(30))*1*3mm}) +(30:3mm) -- +(90:3mm) -- +(150:3mm) -- +(-150:3mm) -- +(-90:3mm) -- +(-30:3mm) -- cycle;
  \draw[fill=lightgray] ({((2*4)-1)*cos(30)*3mm},{(1+sin(30))*1*3mm}) +(30:3mm) -- +(90:3mm) -- +(150:3mm) -- +(-150:3mm) -- +(-90:3mm) -- +(-30:3mm) -- cycle;
  \draw[fill=lightgray] ({((2*5)-1)*cos(30)*3mm},{(1+sin(30))*1*3mm}) +(30:3mm) -- +(90:3mm) -- +(150:3mm) -- +(-150:3mm) -- +(-90:3mm) -- +(-30:3mm) -- cycle;
  \draw[fill=lightgray] ({((2*6)-1)*cos(30)*3mm},{(1+sin(30))*1*3mm}) +(30:3mm) -- +(90:3mm) -- +(150:3mm) -- +(-150:3mm) -- +(-90:3mm) -- +(-30:3mm) -- cycle;
  \draw[fill=lightgray] ({((2*7)-1)*cos(30)*3mm},{(1+sin(30))*1*3mm}) +(30:3mm) -- +(90:3mm) -- +(150:3mm) -- +(-150:3mm) -- +(-90:3mm) -- +(-30:3mm) -- cycle;
  \draw[fill=lightgray] ({((2*8)-1)*cos(30)*3mm},{(1+sin(30))*1*3mm}) +(30:3mm) -- +(90:3mm) -- +(150:3mm) -- +(-150:3mm) -- +(-90:3mm) -- +(-30:3mm) -- cycle;
  \draw[fill=lightgray] ({((2*9)-1)*cos(30)*3mm},{(1+sin(30))*1*3mm}) +(30:3mm) -- +(90:3mm) -- +(150:3mm) -- +(-150:3mm) -- +(-90:3mm) -- +(-30:3mm) -- cycle;
  \draw[fill=lightgray] ({((2*10)-1)*cos(30)*3mm},{(1+sin(30))*1*3mm}) +(30:3mm) -- +(90:3mm) -- +(150:3mm) -- +(-150:3mm) -- +(-90:3mm) -- +(-30:3mm) -- cycle;
  \draw[fill=lightgray] ({((2*0)-2)*cos(30)*3mm},{(1+sin(30))*2*3mm}) +(30:3mm) -- +(90:3mm) -- +(150:3mm) -- +(-150:3mm) -- +(-90:3mm) -- +(-30:3mm) -- cycle;
  \draw[fill=lightgray] ({((2*1)-2)*cos(30)*3mm},{(1+sin(30))*2*3mm}) +(30:3mm) -- +(90:3mm) -- +(150:3mm) -- +(-150:3mm) -- +(-90:3mm) -- +(-30:3mm) -- cycle;
  \draw[fill=lightgray] ({((2*2)-2)*cos(30)*3mm},{(1+sin(30))*2*3mm}) +(30:3mm) -- +(90:3mm) -- +(150:3mm) -- +(-150:3mm) -- +(-90:3mm) -- +(-30:3mm) -- cycle;
  \draw[fill=lightgray] ({((2*3)-2)*cos(30)*3mm},{(1+sin(30))*2*3mm}) +(30:3mm) -- +(90:3mm) -- +(150:3mm) -- +(-150:3mm) -- +(-90:3mm) -- +(-30:3mm) -- cycle;
  \draw[fill=lightgray] ({((2*4)-2)*cos(30)*3mm},{(1+sin(30))*2*3mm}) +(30:3mm) -- +(90:3mm) -- +(150:3mm) -- +(-150:3mm) -- +(-90:3mm) -- +(-30:3mm) -- cycle;
  \draw[fill=lightgray] ({((2*5)-2)*cos(30)*3mm},{(1+sin(30))*2*3mm}) +(30:3mm) -- +(90:3mm) -- +(150:3mm) -- +(-150:3mm) -- +(-90:3mm) -- +(-30:3mm) -- cycle;
  \draw[fill=lightgray] ({((2*6)-2)*cos(30)*3mm},{(1+sin(30))*2*3mm}) +(30:3mm) -- +(90:3mm) -- +(150:3mm) -- +(-150:3mm) -- +(-90:3mm) -- +(-30:3mm) -- cycle;
  \draw[fill=lightgray] ({((2*7)-2)*cos(30)*3mm},{(1+sin(30))*2*3mm}) +(30:3mm) -- +(90:3mm) -- +(150:3mm) -- +(-150:3mm) -- +(-90:3mm) -- +(-30:3mm) -- cycle;
  \draw[fill=lightgray] ({((2*8)-2)*cos(30)*3mm},{(1+sin(30))*2*3mm}) +(30:3mm) -- +(90:3mm) -- +(150:3mm) -- +(-150:3mm) -- +(-90:3mm) -- +(-30:3mm) -- cycle;
  \draw[fill=lightgray] ({((2*9)-2)*cos(30)*3mm},{(1+sin(30))*2*3mm}) +(30:3mm) -- +(90:3mm) -- +(150:3mm) -- +(-150:3mm) -- +(-90:3mm) -- +(-30:3mm) -- cycle;
  \draw[fill=lightgray] ({((2*10)-2)*cos(30)*3mm},{(1+sin(30))*2*3mm}) +(30:3mm) -- +(90:3mm) -- +(150:3mm) -- +(-150:3mm) -- +(-90:3mm) -- +(-30:3mm) -- cycle;
  \draw[fill=lightgray] ({((2*0)-3)*cos(30)*3mm},{(1+sin(30))*3*3mm}) +(30:3mm) -- +(90:3mm) -- +(150:3mm) -- +(-150:3mm) -- +(-90:3mm) -- +(-30:3mm) -- cycle;
  \draw[fill=lightgray] ({((2*1)-3)*cos(30)*3mm},{(1+sin(30))*3*3mm}) +(30:3mm) -- +(90:3mm) -- +(150:3mm) -- +(-150:3mm) -- +(-90:3mm) -- +(-30:3mm) -- cycle;
  \draw[fill=lightgray] ({((2*2)-3)*cos(30)*3mm},{(1+sin(30))*3*3mm}) +(30:3mm) -- +(90:3mm) -- +(150:3mm) -- +(-150:3mm) -- +(-90:3mm) -- +(-30:3mm) -- cycle;
  \draw[fill=lightgray] ({((2*3)-3)*cos(30)*3mm},{(1+sin(30))*3*3mm}) +(30:3mm) -- +(90:3mm) -- +(150:3mm) -- +(-150:3mm) -- +(-90:3mm) -- +(-30:3mm) -- cycle;
  \draw[fill=lightgray] ({((2*4)-3)*cos(30)*3mm},{(1+sin(30))*3*3mm}) +(30:3mm) -- +(90:3mm) -- +(150:3mm) -- +(-150:3mm) -- +(-90:3mm) -- +(-30:3mm) -- cycle;
  \draw[fill=lightgray] ({((2*5)-3)*cos(30)*3mm},{(1+sin(30))*3*3mm}) +(30:3mm) -- +(90:3mm) -- +(150:3mm) -- +(-150:3mm) -- +(-90:3mm) -- +(-30:3mm) -- cycle;
  \draw[fill=lightgray] ({((2*6)-3)*cos(30)*3mm},{(1+sin(30))*3*3mm}) +(30:3mm) -- +(90:3mm) -- +(150:3mm) -- +(-150:3mm) -- +(-90:3mm) -- +(-30:3mm) -- cycle;
  \draw[fill=lightgray] ({((2*7)-3)*cos(30)*3mm},{(1+sin(30))*3*3mm}) +(30:3mm) -- +(90:3mm) -- +(150:3mm) -- +(-150:3mm) -- +(-90:3mm) -- +(-30:3mm) -- cycle;
  \draw[fill=lightgray] ({((2*8)-3)*cos(30)*3mm},{(1+sin(30))*3*3mm}) +(30:3mm) -- +(90:3mm) -- +(150:3mm) -- +(-150:3mm) -- +(-90:3mm) -- +(-30:3mm) -- cycle;
  \draw[fill=lightgray] ({((2*9)-3)*cos(30)*3mm},{(1+sin(30))*3*3mm}) +(30:3mm) -- +(90:3mm) -- +(150:3mm) -- +(-150:3mm) -- +(-90:3mm) -- +(-30:3mm) -- cycle;
  \draw[fill=lightgray] ({((2*10)-3)*cos(30)*3mm},{(1+sin(30))*3*3mm}) +(30:3mm) -- +(90:3mm) -- +(150:3mm) -- +(-150:3mm) -- +(-90:3mm) -- +(-30:3mm) -- cycle;
  \draw[fill=lightgray] ({((2*0)-4)*cos(30)*3mm},{(1+sin(30))*4*3mm}) +(30:3mm) -- +(90:3mm) -- +(150:3mm) -- +(-150:3mm) -- +(-90:3mm) -- +(-30:3mm) -- cycle;
  \draw[fill=lightgray] ({((2*1)-4)*cos(30)*3mm},{(1+sin(30))*4*3mm}) +(30:3mm) -- +(90:3mm) -- +(150:3mm) -- +(-150:3mm) -- +(-90:3mm) -- +(-30:3mm) -- cycle;
  \draw[fill=lightgray] ({((2*2)-4)*cos(30)*3mm},{(1+sin(30))*4*3mm}) +(30:3mm) -- +(90:3mm) -- +(150:3mm) -- +(-150:3mm) -- +(-90:3mm) -- +(-30:3mm) -- cycle;
  \draw[fill=lightgray] ({((2*3)-4)*cos(30)*3mm},{(1+sin(30))*4*3mm}) +(30:3mm) -- +(90:3mm) -- +(150:3mm) -- +(-150:3mm) -- +(-90:3mm) -- +(-30:3mm) -- cycle;
  \draw[fill=lightgray] ({((2*4)-4)*cos(30)*3mm},{(1+sin(30))*4*3mm}) +(30:3mm) -- +(90:3mm) -- +(150:3mm) -- +(-150:3mm) -- +(-90:3mm) -- +(-30:3mm) -- cycle;
  \draw[fill=lightgray] ({((2*5)-4)*cos(30)*3mm},{(1+sin(30))*4*3mm}) +(30:3mm) -- +(90:3mm) -- +(150:3mm) -- +(-150:3mm) -- +(-90:3mm) -- +(-30:3mm) -- cycle;
  \draw[fill=lightgray] ({((2*6)-4)*cos(30)*3mm},{(1+sin(30))*4*3mm}) +(30:3mm) -- +(90:3mm) -- +(150:3mm) -- +(-150:3mm) -- +(-90:3mm) -- +(-30:3mm) -- cycle;
  \draw[fill=lightgray] ({((2*7)-4)*cos(30)*3mm},{(1+sin(30))*4*3mm}) +(30:3mm) -- +(90:3mm) -- +(150:3mm) -- +(-150:3mm) -- +(-90:3mm) -- +(-30:3mm) -- cycle;
  \draw[fill=lightgray] ({((2*8)-4)*cos(30)*3mm},{(1+sin(30))*4*3mm}) +(30:3mm) -- +(90:3mm) -- +(150:3mm) -- +(-150:3mm) -- +(-90:3mm) -- +(-30:3mm) -- cycle;
  \draw[fill=lightgray] ({((2*9)-4)*cos(30)*3mm},{(1+sin(30))*4*3mm}) +(30:3mm) -- +(90:3mm) -- +(150:3mm) -- +(-150:3mm) -- +(-90:3mm) -- +(-30:3mm) -- cycle;
  \draw[fill=lightgray] ({((2*10)-4)*cos(30)*3mm},{(1+sin(30))*4*3mm}) +(30:3mm) -- +(90:3mm) -- +(150:3mm) -- +(-150:3mm) -- +(-90:3mm) -- +(-30:3mm) -- cycle;
  \draw[fill=lightgray] ({((2*0)-5)*cos(30)*3mm},{(1+sin(30))*5*3mm}) +(30:3mm) -- +(90:3mm) -- +(150:3mm) -- +(-150:3mm) -- +(-90:3mm) -- +(-30:3mm) -- cycle;
  \draw[fill=lightgray] ({((2*1)-5)*cos(30)*3mm},{(1+sin(30))*5*3mm}) +(30:3mm) -- +(90:3mm) -- +(150:3mm) -- +(-150:3mm) -- +(-90:3mm) -- +(-30:3mm) -- cycle;
  \draw[fill=lightgray] ({((2*2)-5)*cos(30)*3mm},{(1+sin(30))*5*3mm}) +(30:3mm) -- +(90:3mm) -- +(150:3mm) -- +(-150:3mm) -- +(-90:3mm) -- +(-30:3mm) -- cycle;
  \draw[fill=lightgray] ({((2*3)-5)*cos(30)*3mm},{(1+sin(30))*5*3mm}) +(30:3mm) -- +(90:3mm) -- +(150:3mm) -- +(-150:3mm) -- +(-90:3mm) -- +(-30:3mm) -- cycle;
  \draw[fill=lightgray] ({((2*4)-5)*cos(30)*3mm},{(1+sin(30))*5*3mm}) +(30:3mm) -- +(90:3mm) -- +(150:3mm) -- +(-150:3mm) -- +(-90:3mm) -- +(-30:3mm) -- cycle;
  \draw[fill=lightgray] ({((2*5)-5)*cos(30)*3mm},{(1+sin(30))*5*3mm}) +(30:3mm) -- +(90:3mm) -- +(150:3mm) -- +(-150:3mm) -- +(-90:3mm) -- +(-30:3mm) -- cycle;
  \draw[fill=lightgray] ({((2*6)-5)*cos(30)*3mm},{(1+sin(30))*5*3mm}) +(30:3mm) -- +(90:3mm) -- +(150:3mm) -- +(-150:3mm) -- +(-90:3mm) -- +(-30:3mm) -- cycle;
  \draw[fill=lightgray] ({((2*7)-5)*cos(30)*3mm},{(1+sin(30))*5*3mm}) +(30:3mm) -- +(90:3mm) -- +(150:3mm) -- +(-150:3mm) -- +(-90:3mm) -- +(-30:3mm) -- cycle;
  \draw[fill=lightgray] ({((2*8)-5)*cos(30)*3mm},{(1+sin(30))*5*3mm}) +(30:3mm) -- +(90:3mm) -- +(150:3mm) -- +(-150:3mm) -- +(-90:3mm) -- +(-30:3mm) -- cycle;
  \draw[fill=lightgray] ({((2*9)-5)*cos(30)*3mm},{(1+sin(30))*5*3mm}) +(30:3mm) -- +(90:3mm) -- +(150:3mm) -- +(-150:3mm) -- +(-90:3mm) -- +(-30:3mm) -- cycle;
  \draw[fill=lightgray] ({((2*10)-5)*cos(30)*3mm},{(1+sin(30))*5*3mm}) +(30:3mm) -- +(90:3mm) -- +(150:3mm) -- +(-150:3mm) -- +(-90:3mm) -- +(-30:3mm) -- cycle;
  \draw[fill=lightgray] ({((2*0)-6)*cos(30)*3mm},{(1+sin(30))*6*3mm}) +(30:3mm) -- +(90:3mm) -- +(150:3mm) -- +(-150:3mm) -- +(-90:3mm) -- +(-30:3mm) -- cycle;
  \draw[fill=lightgray] ({((2*1)-6)*cos(30)*3mm},{(1+sin(30))*6*3mm}) +(30:3mm) -- +(90:3mm) -- +(150:3mm) -- +(-150:3mm) -- +(-90:3mm) -- +(-30:3mm) -- cycle;
  \draw[fill=lightgray] ({((2*2)-6)*cos(30)*3mm},{(1+sin(30))*6*3mm}) +(30:3mm) -- +(90:3mm) -- +(150:3mm) -- +(-150:3mm) -- +(-90:3mm) -- +(-30:3mm) -- cycle;
  \draw[fill=lightgray] ({((2*3)-6)*cos(30)*3mm},{(1+sin(30))*6*3mm}) +(30:3mm) -- +(90:3mm) -- +(150:3mm) -- +(-150:3mm) -- +(-90:3mm) -- +(-30:3mm) -- cycle;
  \draw[fill=lightgray] ({((2*4)-6)*cos(30)*3mm},{(1+sin(30))*6*3mm}) +(30:3mm) -- +(90:3mm) -- +(150:3mm) -- +(-150:3mm) -- +(-90:3mm) -- +(-30:3mm) -- cycle;
  \draw[fill=lightgray] ({((2*5)-6)*cos(30)*3mm},{(1+sin(30))*6*3mm}) +(30:3mm) -- +(90:3mm) -- +(150:3mm) -- +(-150:3mm) -- +(-90:3mm) -- +(-30:3mm) -- cycle;
  \draw[fill=lightgray] ({((2*6)-6)*cos(30)*3mm},{(1+sin(30))*6*3mm}) +(30:3mm) -- +(90:3mm) -- +(150:3mm) -- +(-150:3mm) -- +(-90:3mm) -- +(-30:3mm) -- cycle;
  \draw[fill=lightgray] ({((2*7)-6)*cos(30)*3mm},{(1+sin(30))*6*3mm}) +(30:3mm) -- +(90:3mm) -- +(150:3mm) -- +(-150:3mm) -- +(-90:3mm) -- +(-30:3mm) -- cycle;
  \draw[fill=lightgray] ({((2*8)-6)*cos(30)*3mm},{(1+sin(30))*6*3mm}) +(30:3mm) -- +(90:3mm) -- +(150:3mm) -- +(-150:3mm) -- +(-90:3mm) -- +(-30:3mm) -- cycle;
  \draw[fill=lightgray] ({((2*9)-6)*cos(30)*3mm},{(1+sin(30))*6*3mm}) +(30:3mm) -- +(90:3mm) -- +(150:3mm) -- +(-150:3mm) -- +(-90:3mm) -- +(-30:3mm) -- cycle;
  \draw[fill=lightgray] ({((2*10)-6)*cos(30)*3mm},{(1+sin(30))*6*3mm}) +(30:3mm) -- +(90:3mm) -- +(150:3mm) -- +(-150:3mm) -- +(-90:3mm) -- +(-30:3mm) -- cycle;
  \draw[fill=lightgray] ({((2*0)-7)*cos(30)*3mm},{(1+sin(30))*7*3mm}) +(30:3mm) -- +(90:3mm) -- +(150:3mm) -- +(-150:3mm) -- +(-90:3mm) -- +(-30:3mm) -- cycle;
  \draw[fill=lightgray] ({((2*1)-7)*cos(30)*3mm},{(1+sin(30))*7*3mm}) +(30:3mm) -- +(90:3mm) -- +(150:3mm) -- +(-150:3mm) -- +(-90:3mm) -- +(-30:3mm) -- cycle;
  \draw[fill=lightgray] ({((2*2)-7)*cos(30)*3mm},{(1+sin(30))*7*3mm}) +(30:3mm) -- +(90:3mm) -- +(150:3mm) -- +(-150:3mm) -- +(-90:3mm) -- +(-30:3mm) -- cycle;
  \draw[fill=lightgray] ({((2*3)-7)*cos(30)*3mm},{(1+sin(30))*7*3mm}) +(30:3mm) -- +(90:3mm) -- +(150:3mm) -- +(-150:3mm) -- +(-90:3mm) -- +(-30:3mm) -- cycle;
  \draw[fill=lightgray] ({((2*4)-7)*cos(30)*3mm},{(1+sin(30))*7*3mm}) +(30:3mm) -- +(90:3mm) -- +(150:3mm) -- +(-150:3mm) -- +(-90:3mm) -- +(-30:3mm) -- cycle;
  \draw[fill=lightgray] ({((2*5)-7)*cos(30)*3mm},{(1+sin(30))*7*3mm}) +(30:3mm) -- +(90:3mm) -- +(150:3mm) -- +(-150:3mm) -- +(-90:3mm) -- +(-30:3mm) -- cycle;
  \draw[fill=lightgray] ({((2*6)-7)*cos(30)*3mm},{(1+sin(30))*7*3mm}) +(30:3mm) -- +(90:3mm) -- +(150:3mm) -- +(-150:3mm) -- +(-90:3mm) -- +(-30:3mm) -- cycle;
  \draw[fill=lightgray] ({((2*7)-7)*cos(30)*3mm},{(1+sin(30))*7*3mm}) +(30:3mm) -- +(90:3mm) -- +(150:3mm) -- +(-150:3mm) -- +(-90:3mm) -- +(-30:3mm) -- cycle;
  \draw[fill=lightgray] ({((2*8)-7)*cos(30)*3mm},{(1+sin(30))*7*3mm}) +(30:3mm) -- +(90:3mm) -- +(150:3mm) -- +(-150:3mm) -- +(-90:3mm) -- +(-30:3mm) -- cycle;
  \draw[fill=lightgray] ({((2*9)-7)*cos(30)*3mm},{(1+sin(30))*7*3mm}) +(30:3mm) -- +(90:3mm) -- +(150:3mm) -- +(-150:3mm) -- +(-90:3mm) -- +(-30:3mm) -- cycle;
  \draw[fill=lightgray] ({((2*10)-7)*cos(30)*3mm},{(1+sin(30))*7*3mm}) +(30:3mm) -- +(90:3mm) -- +(150:3mm) -- +(-150:3mm) -- +(-90:3mm) -- +(-30:3mm) -- cycle;
  \draw[fill=lightgray] ({((2*0)-8)*cos(30)*3mm},{(1+sin(30))*8*3mm}) +(30:3mm) -- +(90:3mm) -- +(150:3mm) -- +(-150:3mm) -- +(-90:3mm) -- +(-30:3mm) -- cycle;
  \draw[fill=lightgray] ({((2*1)-8)*cos(30)*3mm},{(1+sin(30))*8*3mm}) +(30:3mm) -- +(90:3mm) -- +(150:3mm) -- +(-150:3mm) -- +(-90:3mm) -- +(-30:3mm) -- cycle;
  \draw[fill=lightgray] ({((2*2)-8)*cos(30)*3mm},{(1+sin(30))*8*3mm}) +(30:3mm) -- +(90:3mm) -- +(150:3mm) -- +(-150:3mm) -- +(-90:3mm) -- +(-30:3mm) -- cycle;
  \draw[fill=lightgray] ({((2*3)-8)*cos(30)*3mm},{(1+sin(30))*8*3mm}) +(30:3mm) -- +(90:3mm) -- +(150:3mm) -- +(-150:3mm) -- +(-90:3mm) -- +(-30:3mm) -- cycle;
  \draw[fill=lightgray] ({((2*4)-8)*cos(30)*3mm},{(1+sin(30))*8*3mm}) +(30:3mm) -- +(90:3mm) -- +(150:3mm) -- +(-150:3mm) -- +(-90:3mm) -- +(-30:3mm) -- cycle;
  \draw[fill=lightgray] ({((2*5)-8)*cos(30)*3mm},{(1+sin(30))*8*3mm}) +(30:3mm) -- +(90:3mm) -- +(150:3mm) -- +(-150:3mm) -- +(-90:3mm) -- +(-30:3mm) -- cycle;
  \draw[fill=lightgray] ({((2*6)-8)*cos(30)*3mm},{(1+sin(30))*8*3mm}) +(30:3mm) -- +(90:3mm) -- +(150:3mm) -- +(-150:3mm) -- +(-90:3mm) -- +(-30:3mm) -- cycle;
  \draw[fill=lightgray] ({((2*7)-8)*cos(30)*3mm},{(1+sin(30))*8*3mm}) +(30:3mm) -- +(90:3mm) -- +(150:3mm) -- +(-150:3mm) -- +(-90:3mm) -- +(-30:3mm) -- cycle;
  \draw[fill=lightgray] ({((2*8)-8)*cos(30)*3mm},{(1+sin(30))*8*3mm}) +(30:3mm) -- +(90:3mm) -- +(150:3mm) -- +(-150:3mm) -- +(-90:3mm) -- +(-30:3mm) -- cycle;
  \draw[fill=lightgray] ({((2*9)-8)*cos(30)*3mm},{(1+sin(30))*8*3mm}) +(30:3mm) -- +(90:3mm) -- +(150:3mm) -- +(-150:3mm) -- +(-90:3mm) -- +(-30:3mm) -- cycle;
  \draw[fill=lightgray] ({((2*10)-8)*cos(30)*3mm},{(1+sin(30))*8*3mm}) +(30:3mm) -- +(90:3mm) -- +(150:3mm) -- +(-150:3mm) -- +(-90:3mm) -- +(-30:3mm) -- cycle;
  \draw[fill=lightgray] ({((2*0)-9)*cos(30)*3mm},{(1+sin(30))*9*3mm}) +(30:3mm) -- +(90:3mm) -- +(150:3mm) -- +(-150:3mm) -- +(-90:3mm) -- +(-30:3mm) -- cycle;
  \draw[fill=lightgray] ({((2*1)-9)*cos(30)*3mm},{(1+sin(30))*9*3mm}) +(30:3mm) -- +(90:3mm) -- +(150:3mm) -- +(-150:3mm) -- +(-90:3mm) -- +(-30:3mm) -- cycle;
  \draw[fill=lightgray] ({((2*2)-9)*cos(30)*3mm},{(1+sin(30))*9*3mm}) +(30:3mm) -- +(90:3mm) -- +(150:3mm) -- +(-150:3mm) -- +(-90:3mm) -- +(-30:3mm) -- cycle;
  \draw[fill=lightgray] ({((2*3)-9)*cos(30)*3mm},{(1+sin(30))*9*3mm}) +(30:3mm) -- +(90:3mm) -- +(150:3mm) -- +(-150:3mm) -- +(-90:3mm) -- +(-30:3mm) -- cycle;
  \draw[fill=lightgray] ({((2*4)-9)*cos(30)*3mm},{(1+sin(30))*9*3mm}) +(30:3mm) -- +(90:3mm) -- +(150:3mm) -- +(-150:3mm) -- +(-90:3mm) -- +(-30:3mm) -- cycle;
  \draw[fill=lightgray] ({((2*5)-9)*cos(30)*3mm},{(1+sin(30))*9*3mm}) +(30:3mm) -- +(90:3mm) -- +(150:3mm) -- +(-150:3mm) -- +(-90:3mm) -- +(-30:3mm) -- cycle;
  \draw[fill=lightgray] ({((2*6)-9)*cos(30)*3mm},{(1+sin(30))*9*3mm}) +(30:3mm) -- +(90:3mm) -- +(150:3mm) -- +(-150:3mm) -- +(-90:3mm) -- +(-30:3mm) -- cycle;
  \draw[fill=lightgray] ({((2*7)-9)*cos(30)*3mm},{(1+sin(30))*9*3mm}) +(30:3mm) -- +(90:3mm) -- +(150:3mm) -- +(-150:3mm) -- +(-90:3mm) -- +(-30:3mm) -- cycle;
  \draw[fill=lightgray] ({((2*8)-9)*cos(30)*3mm},{(1+sin(30))*9*3mm}) +(30:3mm) -- +(90:3mm) -- +(150:3mm) -- +(-150:3mm) -- +(-90:3mm) -- +(-30:3mm) -- cycle;
  \draw[fill=lightgray] ({((2*9)-9)*cos(30)*3mm},{(1+sin(30))*9*3mm}) +(30:3mm) -- +(90:3mm) -- +(150:3mm) -- +(-150:3mm) -- +(-90:3mm) -- +(-30:3mm) -- cycle;
  \draw[fill=lightgray] ({((2*10)-9)*cos(30)*3mm},{(1+sin(30))*9*3mm}) +(30:3mm) -- +(90:3mm) -- +(150:3mm) -- +(-150:3mm) -- +(-90:3mm) -- +(-30:3mm) -- cycle;
  \draw[fill=lightgray] ({((2*0)-10)*cos(30)*3mm},{(1+sin(30))*10*3mm}) +(30:3mm) -- +(90:3mm) -- +(150:3mm) -- +(-150:3mm) -- +(-90:3mm) -- +(-30:3mm) -- cycle;
  \draw[fill=lightgray] ({((2*1)-10)*cos(30)*3mm},{(1+sin(30))*10*3mm}) +(30:3mm) -- +(90:3mm) -- +(150:3mm) -- +(-150:3mm) -- +(-90:3mm) -- +(-30:3mm) -- cycle;
  \draw[fill=lightgray] ({((2*2)-10)*cos(30)*3mm},{(1+sin(30))*10*3mm}) +(30:3mm) -- +(90:3mm) -- +(150:3mm) -- +(-150:3mm) -- +(-90:3mm) -- +(-30:3mm) -- cycle;
  \draw[fill=lightgray] ({((2*3)-10)*cos(30)*3mm},{(1+sin(30))*10*3mm}) +(30:3mm) -- +(90:3mm) -- +(150:3mm) -- +(-150:3mm) -- +(-90:3mm) -- +(-30:3mm) -- cycle;
  \draw[fill=lightgray] ({((2*4)-10)*cos(30)*3mm},{(1+sin(30))*10*3mm}) +(30:3mm) -- +(90:3mm) -- +(150:3mm) -- +(-150:3mm) -- +(-90:3mm) -- +(-30:3mm) -- cycle;
  \draw[fill=lightgray] ({((2*5)-10)*cos(30)*3mm},{(1+sin(30))*10*3mm}) +(30:3mm) -- +(90:3mm) -- +(150:3mm) -- +(-150:3mm) -- +(-90:3mm) -- +(-30:3mm) -- cycle;
  \draw[fill=lightgray] ({((2*6)-10)*cos(30)*3mm},{(1+sin(30))*10*3mm}) +(30:3mm) -- +(90:3mm) -- +(150:3mm) -- +(-150:3mm) -- +(-90:3mm) -- +(-30:3mm) -- cycle;
  \draw[fill=lightgray] ({((2*7)-10)*cos(30)*3mm},{(1+sin(30))*10*3mm}) +(30:3mm) -- +(90:3mm) -- +(150:3mm) -- +(-150:3mm) -- +(-90:3mm) -- +(-30:3mm) -- cycle;
  \draw[fill=lightgray] ({((2*8)-10)*cos(30)*3mm},{(1+sin(30))*10*3mm}) +(30:3mm) -- +(90:3mm) -- +(150:3mm) -- +(-150:3mm) -- +(-90:3mm) -- +(-30:3mm) -- cycle;
  \draw[fill=lightgray] ({((2*9)-10)*cos(30)*3mm},{(1+sin(30))*10*3mm}) +(30:3mm) -- +(90:3mm) -- +(150:3mm) -- +(-150:3mm) -- +(-90:3mm) -- +(-30:3mm) -- cycle;
  \draw[fill=lightgray] ({((2*10)-10)*cos(30)*3mm},{(1+sin(30))*10*3mm}) +(30:3mm) -- +(90:3mm) -- +(150:3mm) -- +(-150:3mm) -- +(-90:3mm) -- +(-30:3mm) -- cycle;
\end{tikzpicture}
%
\]

\begin{theorem*}[Hex theorem]
  If evey hexagon on the board is coloured black or white, then there is a path of black hexagons linking the two black
  sides or a path of white hexagons linking the two white sides.
\end{theorem*}

The following proof is due to Gale \cite[pp.~820--2]{gale_hex}, but is presented in a minimal form here:

\begin{proof}[Spatial proof]
  Define a path as follows. The path enters the board at the south, with the black side of the board on its left and
  white side of the board on its right. Inductively extend the path as so that there is always black on the left and
  white on the right: this is possible since whenever one reaches a junction of edges of hexagons, the hexagon or the
  part of the side of the board that lies directly ahead is either black (in which case the path turns right) or white
  (in which case the path turns left). The path thus defined does not intersect itself itself, since the first
  intersection would have to involve all three edges incident to a given point, at least one of which has the same
  colour on each side, which would contradict the path's definition. Since there are finitely many edges, the path must
  therefore exit the board at a join between the sides. The path always has black on the left and white on the right, so
  must exit either at the east (in which case the black hexagons on its left link the two black sides of the board) or
  the west (in which case the white hexagons on its right side link the two white sides of the board).
\end{proof}

This proof involves no diagram, but it does appeal to spatial thinking, in particular when the reader must visualize the
how the hexagons along one side of the path link the black or white sides. This appeal to spatial thinking has the merit
of leading to a much simpler proof than the previously-known one, which invoked the Jordan Curve Theorem; see the
discussion by Gale \cite[820]{gale_hex}. This is the first gain in simplicity supplied by spatial thinking:
\emph{avoidance of invocation of technical results.}

\section{Spatial reasoning and induction}
\label{sec:basicpictureproofs}

One of the canonical examples of a \enquote{proof without words} \cite[p.~60]{nelsen_proofswithoutwords1} and of the use of spatial thinking in proving a
numerical result, is of the formula for the sum of the first $n$ integers:

\begin{theorem*}
  $\sum_{i=1}^n i = \frac{n^2}{2} + \frac{n}{2}$.
\end{theorem*}

\begin{proof}[Spatial proof]
  
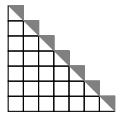
\begin{equation}
    \label{eq:sumoffirstn}
    \begin{tikzpicture}[x=2mm,y=2mm,baseline=9mm]
      \useasboundingbox (-1,1) -- (-8,8);
      \foreach \x in {1,...,7} {
        \foreach \y in {1,...,\x} {
          \draw (-\x,\y) -- +(-1,0) -- +(-1,1) -- +(0,1) -- cycle;,
        }
      };
      \clip (0,0) -- (-9,9) -| cycle;
      \foreach \x in {1,...,7} {
        \filldraw[gray] (-\x,\x) -- +(-1,0) -- +(-1,1) -- +(0,1) -- cycle;
      };
    \end{tikzpicture}
  \end{equation}
\end{proof}

\begin{proof}[Non-spatial proof, abbreviated]
  Proceed by induction on $n$:
  \begin{align*}
    \text{Base of induction:}\quad \sum_{i=1}^1 i &= 1 = \frac{1^2}{2}+\frac{1}{2}.\\
    \text{Inductive hypothesis:}\quad \sum_{i=1}^n i &= \frac{n^2}{2} + \frac{n}{2} \\
    \text{Inductive step:}\quad \sum_{i=1}^{\mathclap{n+1}} i &= \parens[\Big]{\sum_{i=1}^n i} + (n+1) \\
                                                  &= \frac{n^2}{2} + \frac{n}{2} + \frac{2n}{2} + \frac{2}{2} \\
                                                  &= \frac{(n+1)^2}{2} + \frac{n+1}{2}. \qedhere
  \end{align*}
\end{proof}

How does the spatial proof \enquote{work}? Brown \cite[pp.~40--1]{brown_philosophy} notes that mathematician and
philosophers divided into two roughly equal groups. One group holds that the diagram \emph{encodes} an induction, and is
thus a legitimate proof. The other group holds that the diagram is a heuristic device that \emph{suggests} an induction,
and that it is this suggested induction that is a legitimate proof. Brown himself holds that there is no induction
present, arguing that \textquote[{\cite[p.~40]{brown_philosophy}}]{Some \enquote{pictures} \ldots\ are windows to
  Plato's heaven}, or, less poetically but perhaps more accurately, that \textquote[{\cite[p.~40]{brown_philosophy}}]{As
  telescopes help the unaided eye, so some diagrams are instruments \ldots\ which help the unaided mind's eye.}

Do these opposing views have implications of simplicity in proof? Let us accept that the picture supplies a valid proof,
whether by encoding or suggesting an induction or by aiding the mind's eye. \emph{With this proviso}, I think that few
would disagree that the spatial proof is simpler than the inductive proof.

Let us consider first the possibility that the diagram encodes or suggests an induction. Chihara
\cite[pp.~302--3]{chihara_structural} describes the thought-process that the diagram triggers. In abbreviated form, it
is as follows: When one examines the diagram, one sees that to go from the $n$-th to the $(n+1)$-th row, one has to duplicate
the $n$-th row and add another square, and so there are $n+1$ squares in the $(n+1)$-th row.  Thus, \textquote{an
  intuitive version of mathematical induction} allows one to see that for any $n$, the number of squares in the $n$-th
row will be $n$. Therefore calculating $\sum_{i=1}^n i$ requires calculating the total number of squares in the diagram,
which is $n^2\!/2 + n/2$, by the formula for the area of a triangle plus the number of squares supplied by the shaded
half-squares.

One can view this thought-process as extracting an induction encoded in the diagram, or as constructing an induction
suggested by it. Whichever view one holds not seem relevant for simplicity: the point is that \emph{if} the proof is
through an induction triggered by the diagram, it must be simpler than the induction in the non-spatial proof. Of
course, the spatial proof does not require any notation and is simpler in that respect. But is it simpler in some other
way? Possibly the diagram serves to guide the reader through the induction in a way that is simpler than the non-spatial
proof. But drawing this inference seems unwarranted, for the two inductions are fundamentally different, and end at
different points. The induction in the non-spatial proof leads directly to the statement of the theorem. The induction
triggered by the diagram shows that the $1$ square plus $2$ squares plus \ldots plus $n$ squares form what may be
described as a \enquote{step-triangular shape} of side $n$, which is a geometric statement, and it is from this
geometric statement that the result is deduced.

Now return to the possibility suggested by Brown that the diagram aids the mind's eye, without any inductive process. In
this case, I suggest that the thought-process is that one immediately perceives that the partially shaded squares form a
diagonal, and that two conclusions follow from this immediately: the diagram is $n$ squares wide and $n$ squares high,
the particular number $n$ in the actual diagram never being noticed. Therefore the lengths of the rows are
$1,2,\ldots,n$. Calculating $\sum_{i=1}^n i$ thus requires calculating the total number of squares in the diagram, which
is $n^2/2 + n/2$, by the formula for the area of a triangle plus the number of squares supplied by the shaded
half-squares, as in Chihara's description above. Since the particular number $n$ in the diagram remained unnoticed, this holds
for all $n$.

The simplicity here seems to arise from the fact that the perception of the diagonal is immediate and the conclusions
from it are almost built into the geometric perception. That is, it is the absence of induction and its replacement with
spatial thinking (or, if one prefers, the view of Plato's heaven) that makes the picture proof simpler.

In summary, this discussion has suggested several ways in which spatial reasoning leads to gains in simplicity: first,
\emph{decreased complexity (or elimination) of notation}; second, depending on the position one takes on the triggering
of an induction, either an \emph{easier-to-understand induction} or an \emph{avoidance of induction}. However, it could
be argued that in the \enquote{non-inductive} thought-process I described above there could be a hidden induction that a
mathematically-trained reader can do \enquote{without thinking} and that an untrained reader could intuit. A more
convincing example of avoidance of induction will be considered in \fullref{Section}{sec:inversions} below.

\section{Definitions and inferences}
\label{sec:ferrers}

This section and the following ones consider spatial proofs involving partitions of numbers, using Ferrers
diagrams. Only the background necessary for understanding the proofs below is recalled here; see \cite{andrews_partitions}
for further reading.

A \defterm{partition} of a natural number (that is, positive integer) $n$ is a finite nonincreasing sequence of natural
numbers $\lambda = \parens{\lambda_1,\lambda_2,\ldots,\lambda_k}$ that sums to $n$; the $\lambda_i$ are the
\defterm{parts} of the partition. For example, a partition of $24$ is $\parens{8,6,6,3,1}$. A \defterm{Ferrers diagram}
or \defterm{Ferrers graph} is a visual representation of such a partition as an array of dots, left-aligned, with the
$i$-th row (counting from the top) containing $\lambda_i$ dots. For example, the following Ferrers diagram represents
the partition $\parens{8,6,6,3,1}$:
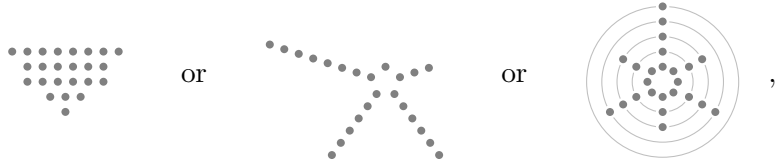
\[
  \begin{tikzpicture}[x=2mm,y=-2mm,baseline=-4mm]
    \foreach \y/\xsize in {1/8,2/6,3/6,4/3,5/1} {
      \foreach \x in {1,...,\xsize} {
        \fill[gray] (\x,\y) circle (1.5pt);
      }
    }
  \end{tikzpicture}
\]
Since a partition is a non-increasing sequence, the rows of a Ferrers diagram are non-increasing in length from top to
bottom, and a left-aligned array of dots forms a Ferrers diagram (that is, represents a partition) precisely when this
condition is satisfied.

Finite nonincreasing sequences of of natural numbers and Ferrers diagrams are thus alternative representations of
partitions. As Starikova \cite{starikova_why} has pointed out, alternative representations of mathematical objects can
certainly ease the doing of mathematics, not just in proof, but in discovery and in formulation of new concepts.

A fundamental result in partition theory, due to Euler, is the following:

\begin{theorem*}[{\cite[Theorem~1.4]{andrews_partitions}}]
  The number of partitions of $n$ with at most $m$ parts equals the number of partitions of $n$ in which no part exceeds
  $m$.
\end{theorem*}

\begin{proof}[Spatial proof]
  Reflect the Ferrers diagram of a partition along its main diagonal:
  \begin{equation}
    \label{eq:conjugacyferrers}
    \begin{tikzpicture}[x=2mm,y=-2mm,baseline=-4mm]
      \foreach \y/\xsize in {1/8,2/6,3/6,4/3,5/1} {
        \foreach \x in {1,...,\xsize} {
          \fill[gray] (\x,\y) circle (1.5pt);
        }
      }
    \end{tikzpicture}
    \quad\mapsto\quad
    \begin{tikzpicture}[x=2mm,y=-2mm,baseline=-4mm]
      \foreach \y/\xsize in {1/8,2/6,3/6,4/3,5/1} {
        \foreach \x in {1,...,\xsize} {
          \fill[gray] (\y,\x) circle (1.5pt);
        }
      }
    \end{tikzpicture}
  \end{equation}
  This gives a bijection between the two classes of partitions.
\end{proof}

This bijection could be defined non-spatially as
\begin{equation}
  \label{eq:conjugacyformal}
  (\lambda_1,\lambda_2,\ldots,\lambda_k) \mapsto (\mu_1,\mu_2,\ldots,\mu_p), \quad\left\{\parbox{40mm}{\raggedright where $\mu_i$ is the number of $\lambda_j$ that are greater than $i$.}\right..
\end{equation}
However, this definition is less simple. First, the reader grasps the spatial definition \enquote{reflecting along the
  diagonal} immediately, whereas one has to pause and think to understand \eqref{eq:conjugacyformal}. Second, there are
two inferences that are immediate from the spatial definition, but require some thought to deduce from
\eqref{eq:conjugacyformal}. The first inference is that the map is well-defined, in the sense that the image is a
partition, and indeed a partition \emph{of the same natural number $n$}. This is not \emph{difficult} to prove from \eqref{eq:conjugacyformal},
but is trivial from the spatial definition, since reflection preserves both the non-increasing lengths of lines and the
total number of dots. The second inference is that the map is a bijection (and indeed an involution): trivial from the
spatial definition, but requiring some work to prove from \eqref{eq:conjugacyformal}.

This is another gain in simplicity from spatial reasoning: \emph{simpler definitions of transformations using visual
  representations}, and \emph{simpler inference of properties of those transformations} (such as bijectivity).  Notice
that the particular choice of visual representation using Ferrers diagrams that enables the definition of the bijection
using reflection in \eqref{eq:conjugacyferrers}. Different visual representations are possible, such as representing
the partition $\parens{8,6,6,3,1}$ as
\[
  \begin{tikzpicture}[x=2mm,y=-2mm,baseline=-6mm]
    \foreach \y/\xsize in {1/8,2/6,3/6,4/3,5/1} {
      \foreach \x in {1,...,\xsize} {
        \fill[gray] ({\x-(\xsize/2)},\y) circle (1.5pt);
      }
    }
  \end{tikzpicture}
  \qquad
  \text{or}
  \qquad
  \begin{tikzpicture}[baseline=0mm]
    \foreach \x/\ysize in {1/8,2/6,3/6,4/3,5/1} {
      \foreach \y in {1,...,\ysize} {
        \fill[gray] ({(360/5)*\x+90}:{(\y)*2mm}) circle (1.5pt);
      };
    };
  \end{tikzpicture}
  \qquad
  \text{or}
  \qquad
  \begin{tikzpicture}[baseline=0mm]
    \foreach \x/\ysize in {1/8,2/6,3/6,4/3,5/1} {
      \draw[lightgray] (0,0) circle ({(\x)*2mm});
      \foreach \y in {1,...,\ysize} {
        \fill[white] ({(360/\ysize)*\y+90}:{(\x)*2mm}) circle (2.5pt);
        \fill[gray] ({(360/\ysize)*\y+90}:{(\x)*2mm}) circle (1.5pt);
      };
    };
  \end{tikzpicture}
  \quad,
\]
but none of these representations gives rise to such a simple definition of the bijection as
\eqref{eq:conjugacyferrers}, nor to such simple inference of its properties.

\section{Ferrers diagrams in a more spatial proof}
\label{sec:inversions}

Let $\ell,m,n \in \nset$. Let $\xi_1\xi_2\cdots \xi_{\ell+m}$ be a permutation of $1^\ell2^m$. An \defterm{inversion} in
this permutation is a pair $(i,j)$ such that $i < j$ and $\xi_i = 2$ and $\xi_j = 1$. (That is, a symbol $1$
  after a symbol $2$.) Let $\Inv(m,\ell;n)$ denote the set of permutations $\xi_1\xi_2\cdots x_{\ell+m}$ of $1^m2^\ell$
that have exactly $n$ inversions. Let $P(m,\ell;n)$ denote the set of partitions of $n$ into $\ell$ parts, each at most
$m$.

\begin{theorem*}[{\cite[Theorem 3.5]{andrews_partitions}}]
  $\abs[\big]{P(\ell,m;n)} = \abs[\big]{\Inv(\ell,m;n)}$.
\end{theorem*}

The following proof is a modified version of \cite[Proof of Theorem 3.5]{andrews_partitions}.

\begin{proof}[Spatial proof]
  The aim is to define an explicit bijection between the set of partitions $P(m,\ell;n)$ and the set of permutations
  $\Inv(m,\ell;n)$.

  Consider a partition of $n$ with at most $\ell$ parts, each at most $m$. Place an $m \times \ell$ grid around its Ferrers
  diagram and \enquote{wrap} the diagram with a path following the grid, as follows:
  
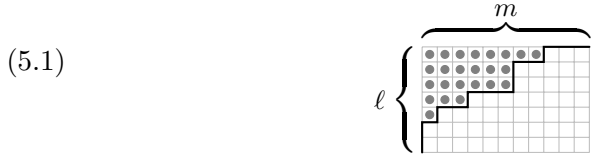
\begin{equation}
    \label{eq:wrapping}
    \begin{tikzpicture}[x=2mm,y=-2mm,baseline=-4mm]
      \foreach \y/\xsize in {1/8,2/6,3/6,4/3,5/1} {
        \foreach \x in {1,...,\xsize} {
          \filldraw[gray] (\x,\y) circle (1.5pt);
        }
      }
      \begin{scope}[shift={(1mm,-1mm)}]
        \draw[lightgray] (0,0) grid[xstep=2mm,ystep=2mm] (11,7);
        \draw[thick,line cap=round] (0,7) -- (0,5) -- (1,5) -- (1,4) -- (3,4) -- (3,3) -- (6,3) -- (6,1) -- (8,1) -- (8,0) -- (11,0);
        \node[anchor=south,inner sep=0mm] (topbrace) at (5.5,-.4) {$\overbrace{\hbox{\vrule width 22mm height 0mm depth 0mm}}$};
        \node[anchor=south,font=\small,inner sep=0mm,yshift=.6mm] at (topbrace.north) {$m$};
        \node[anchor=south,rotate=90] (sidebrace) at (.2,3.5) {$\overbrace{\vrule width 14mm height 0mm depth 0mm}$};
        \node[anchor=east,font=\small,inner sep=0mm] at (sidebrace.north) {$\ell$};
      \end{scope}
    \end{tikzpicture}
    \
  \end{equation}
  Trace the steps made by path through the grid starting at the upper-right and moving leftwards and downwards. If the
  path makes a downwards step, write $2$. If it makes a leftward step, write $1$. The resulting sequence contains $m$
  symbols $1$ and $\ell$ symbols $2$, and so is a permutation of $1^m2^\ell$.

  Note that there is a one-to-one correspondence between dots of the Ferrers diagram and symbols $1$ (leftward steps)
  after symbols $2$ (downward steps):
  
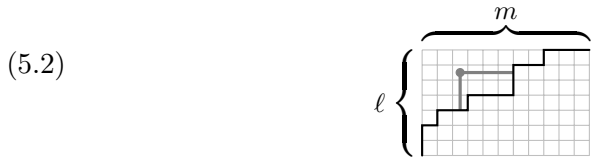
\begin{equation}
    \label{eq:inversions}
    \begin{tikzpicture}[x=2mm,y=-2mm,baseline=-4mm]
      \filldraw[gray] (3,2) circle (1.5pt);
      \begin{scope}[shift={(1mm,-1mm)}]
        \draw[lightgray] (0,0) grid[xstep=2mm,ystep=2mm] (11,7);
      \end{scope}
      \filldraw[gray,very thick] (3,2) rectangle (3,4.5);
      \filldraw[gray,very thick] (3,2) rectangle (6.5,2);
      \begin{scope}[shift={(1mm,-1mm)}]
        \draw[thick,line cap=round] (0,7) -- (0,5) -- (1,5) -- (1,4) -- (3,4) -- (3,3) -- (6,3) -- (6,1) -- (8,1) -- (8,0) -- (11,0);
        \node[anchor=south,inner sep=0mm] (topbrace) at (5.5,-.4) {$\overbrace{\hbox{\vrule width 22mm height 0mm depth 0mm}}$};
        \node[anchor=south,font=\small,inner sep=0mm,yshift=.6mm] at (topbrace.north) {$m$};
        \node[anchor=south,rotate=90] (sidebrace) at (.2,3.5) {$\overbrace{\vrule width 14mm height 0mm depth 0mm}$};
        \node[anchor=east,font=\small,inner sep=0mm] at (sidebrace.north) {$\ell$};
      \end{scope}
    \end{tikzpicture}
    \
  \end{equation}
  So there are exactly $n$ inversions in this permutation.

  Since there is a one-to-one correspondence between partitions and paths, this defines a bijection between the
  permutations of $1^m 2^\ell$ with exactly $n$ inversions and the number of partitions of $n$ with at most $\ell$ parts,
  each of size at most $m$. This completes the proof.
\end{proof}

\subsection{Comparison of simplicity relative to a non-spatial proof}

Spatial thinking is used at the following steps of the proof:
\begin{itemize}
\item[(S1)] The definition of the path that wraps a Ferrers diagram of the partition. (Note that the definition is
  incomplete without the diagram. In particular, the reader is expected to understand the meaning of \enquote{wrap}
  directly from the diagram.)


\item[(S2)] The definition of the sequence determined by the path, which uses the visual idea of travelling along the path,
  and of making leftward or downward steps.
\item[(S3)] Deduction that this sequence is a permutation of $1^m 2^\ell$, which is immediate from the correspondence of
  symbols $1$ and $2$ to (respectively) leftward and downward steps, and the horizontal and vertical distance that the
  path must cover.
\item[(S4)] Deduction that there one-to-one correspondence between inversions in the permutation and points of $n$, which
  follows from the spatial arrangement of the dots in the Ferrers diagram.
\item[(S5)] Deduction that there is a one-to-one correspondence between partitions and paths, which is immediate because
  the Ferrers diagram is seen to precisely fill the path, so that a path determines the corresponding Ferrers diagram
  and vice versa.
\end{itemize}

A non-spatial syntactic proof could proceed (in outline) as follows:
\begin{itemize}
\item[(N1)] Definition of a map $\phi$ taking partitions to sequences, given by
  \begin{equation}
    \label{eq:syntacticsequence}
    \phi(\lambda) = 1^{m-\lambda_1}21^{\lambda_1-\lambda_2}21^{\lambda_2-\lambda_3}2 \cdots 1^{\lambda_{k-1}-\lambda_k}21^{\lambda_k}2^{\ell-k}.
  \end{equation}
\item[(N2)] Deduction that $\phi(\lambda)$ is a permutation of $1^m2^\ell$, which requires an explicit
  calculation of the total number of symbols $1$ and $2$ in the sequence defined by \eqref{eq:syntacticsequence}.
\item[(N3)] Deduction that $\phi(\lambda)$ contains exactly $n$ inversions, which again requires an explicit calculation
  using the definition \eqref{eq:syntacticsequence}.
\item[(N4)] Deduction that there is a one-to-one correspondence between $P(m,\ell;n)$ and $\Inv(m,\ell;n)$:
  \begin{itemize}
  \item[(N4a)] Proof of injectivity: given $\phi(\lambda)$, the tail of symbols $2$ at the end determines $k$; the
    string of symbols $1$ immediately to the right of the $k$-th symbol $2$ determines $\lambda_k$; by induction,
    knowing $\lambda_{i+1}$, the string of symbols $1$ immediately to the right of the $i$-th symbol $2$ determines
    $\lambda_i$.
  \item[(N4b)] Proving that the image of $P(m,k;n)$ is $\Inv(m,k;n)$: a notationally messy argument producing a
    partition mapping onto a given permuation of $1^m2^\ell$ with $n$ inversions.
  \end{itemize}
\end{itemize}

Notice that there is a broad correspondence in overall structure between the given spatial proof and the non-spatial
proof outlined above:
\begin{align*}
  \left.
  \begin{array}{r}
    \text{(S1)} \\
    \text{(S2)}
  \end{array}
  \right\} &\longleftrightarrow \text{(N1)}, \qquad
  \text{(S3)} \longleftrightarrow \text{(N2)}, \\
  \text{(S4)} &\longleftrightarrow \text{(N3)}, \qquad
  \text{(S5)} \longleftrightarrow \text{(N4)}
  \begin{cases}
    \text{(N4a)} \\
    \text{(N4b)}
  \end{cases}.
\end{align*}
Considered on this level of overall structure, therefore, there is little difference in the simplicity of the two
proofs. However, within corresponding steps, there are major differences in simplicity. The definition (N1) is
notationally less simple than (S1)--(S2), and in particular in (N1) it has been necessary to introduce notation for a
partition and its parts. Steps (S3) and (S4) seem simpler than (N2) and (N3): the former are immediate on making certain
visual observations, while the latter require explicit (though admittedly not difficult) calculations. Most
importantly, step (S5) is simpler, again being immediate, but (N4) in particular uses an induction.

There seem to be two gains in simplicity from spatial thinking here: \emph{simpler definitions of maps using visual
  representations}, and \emph{simpler inference of properties of those maps}. This is to be distinguished from the gains
in simplicity related to transformations discussed in \fullref{Section}{sec:ferrers}. Here, there is greater simplicity
even though the visual representation is being used to define, not another visual representation of [a mathematical
object of] the same species, but a sequence.

The contrast between use of induction in step (N4) but nowhere in the visual proof requires further attention since, as
discussed above in \fullref{Section}{sec:basicpictureproofs}, there is disagreement over whether the diagram
\eqref{eq:sumoffirstn} triggers an induction (whether by encoding or suggestion). It is thus necessary to consider the
possibility that, at least for some readers, \eqref{eq:wrapping} triggers an induction.

There is an obvious contrast with \eqref{eq:sumoffirstn}, because the increasing length of rows in
\eqref{eq:sumoffirstn} corresponds to the statement one wishes to deduce. That is, the visualization (of how additional
rows increasing in length) parallels the induction (on the number of rows). There seems no such obvious correspondence
in \eqref{eq:wrapping}. At this point, the proof considers the Ferrers diagram of a partition of $n$ with at most $\ell$
parts, each at most $m$. An induction on $\ell$ and/or $m$ corresponds to extending the grid to the right or
downwards. Certainly, the reader is called upon to visualize an $m \times \ell$ grid of for arbitrary $\ell$ and $m$,
but no conclusion is drawn directly from this in the way that the conclusion about lengths of rows is drawn in
\eqref{eq:wrapping}. One could consider induction on $n$, or the number of parts, or on the lengths of the parts. But
this would involve assuming the existence of a unique path wrapping around the Ferrers diagram and satisfying certain
properties, then adding one or more squares and deducing the existence of a unique path wrapping aroung the diagram and
satisfying certain other properties. (One hesitates to try to write down such an induction.) However, this
seems very far from the spatial thinking of wrapping the Ferrers diagram. So this seems to be a rather more certain
instance of spatial thinking leading to a gain in simplicity by \emph{avoidance of induction} than the proof discussed in
\fullref{Section}{sec:basicpictureproofs}.

Returning to the relative simplicity of the two proofs: the spatial proof is also simpler because it has \emph{greater
  unity}: the definition of the path and all but one deduction step appeal to spatial thinking by calling on the same
diagram. The remaining step deductive step appeals to spatial thinking about a second diagram that is very closely
related to the first. Once could even combine the diagrams, but this requires using a diagram for (S1)--(S3) that
contains information that is not relevant until (S4). This would arguably be a decrease in simplicity; it certainly
seems to be a decrease in elegance. This increased unity is perhaps linked with \emph{better surveyability} of the
proof: the unity increases the ease with which \textquote[{\cite[p.~59]{tymoczko_4colourphil} (emphasis in
  original)}]{[t]he mathematician \emph{surveys} the proof in its entirety and thereby comes to \emph{know} the
  conclusion}.

\subsection{Between the proofs}

In (S1), the path is defined by example. One could avoid this and use an explicit definition:
\begin{itemize}
\item[(T1)] For the partition $\lambda = \parens{\lambda_1,\ldots,\lambda_k}$ the path is the map
  $\pi_\lambda : [0,1] \to \rset^2$ defined by\ldots,
\end{itemize}
giving explicit coordinates for the images of each point $\pi_\lambda(t)$. This still defines the path as an object in
the Euclidean plane, but without requiring spatial thinking of the reader. (Although many readers would doubtless try to
visualize the definition.) Is (T1) more or less simple than (S1)? On the one hand, (T1) has introduced a new concept, in
the form of an explicit coordinate system; on the other hand, it has eliminated a concept, in the form of the Ferrers
diagram.

There could conceivably be many proofs beginning with (T1): here only one is considered. Suppose the next steps are:
\begin{itemize}
  \item[(T2)] Definition of a sequence from the path $\pi_\lambda$.
  \item[(T3)] Deduction of what the sequence is in terms of $\lambda$.
\end{itemize}
The conclusion of (T3) would be \eqref{eq:syntacticsequence}; the proof could then proceed via (N2)--(N4). Now, (T1) was
an attempt to explicitly define the path used in (S1) without using spatial thinking. Thus the proof (T1)--(T3),
(N2)--(N4) can be thought of as between the spatial proof (S1)--(S5) and the non-spatial proof (N1)--(N4). However, it
is certainly less simple than either the spatial or the non-spatial proofs. It shares the explicit calculations and
induction of the non-spatial proof, and yet it uses the path $\pi_\lambda$ to do what (N1) does by definition.

Thus a na\"{\i}ve attempt to \enquote{de-spatialize} the original spatial proof may result in a decrease in simplicity,
even if there is another non-spatial proof that is of greater simplicity in than the resulting
\enquote{de-spatialized} one. (Of course, in this particular case, a hypothetical mathematician who had formulated
(T1)--(T3) would doubtless notice that all mention of the paths $\pi_\lambda$ could be eliminated and would obtain the
non-spatial proof (N1)--(N4).)

\section{Spatial thinking and exhaustion of cases}

Euler's pentagonal number theorem is the following identity:

\begin{theorem*}
\[
  \prod_{m=1}^\infty (1-x^m) =
  1 + \sum_{n=0}^\infty c_nx^n,\quad
  \text{where }c_n =
  \begin{cases}
    0 & \text{if $n \neq k(3k\pm 1)/2$,} \\
    (-1)^k & \text{if $n = k(3k\pm 1)/2$,}
  \end{cases}
\]
\end{theorem*}

It is related to partitions because it can be shown that $c_n = p_{\mathrm{e}}(n) - p_{\mathrm{o}}(n)$, where
$p_{\mathrm{e}}(n)$ and $p_{\mathrm{o}}(n)$ are, respectively, the numbers of partitions of $n$ into an even number and
into an odd number of unequal parts. The combinatorial proof by Franklin \cite{franklin_euler}, as explained by Hardy
\cite[\S~6.2]{hardy_ramanujan} and Hardy \& Wright \cite[\S~19.11]{hardy_introduction}, uses transformations of Ferrers
diagrams establish a near-correspondence between partitions of $n$ into an even number and into an odd number of unequal
parts, with exactly correspondence or an excess of one partition of each type giving rise to the cases where
$p_{\mathrm{e}}(n) - p_{\mathrm{o}}(n)$ is $0$, $1$, or $-1$.

\begin{proof}[Spatial proof, key steps only]
  Define the \defterm{base} to be the bottommost row of the diagram, and define the \defterm{slope} to be the longest
  line of dot starting at the top-rightmost and moving towards the bottom-left, as in the following diagram:
  
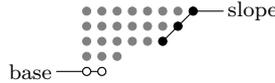
\[
    \begin{tikzpicture}[x=2mm,y=-2mm,baseline=-4mm]
      \foreach \y/\xsize in {1/8,2/7,3/6,4/3,5/2} {
        \foreach \x in {1,2,...,\xsize} {
          \filldraw[gray] (\x,\y) circle (1.5pt);
        }
      }
      \node[font=\scriptsize,inner sep=.5mm,anchor=east] (beta) at (-1,5) {base};
      \draw[black] (beta) -- (1,5);
      \draw[black] (1,5) -- (2,5);
      \draw[black,fill=white] (1,5) circle (1.5pt);
      \draw[black,fill=white] (2,5) circle (1.5pt);
      \node[font=\scriptsize,inner sep=.5mm,anchor=west] (sigma) at (10,1) {slope};
      \draw[black] (sigma) -- (8,1);
      \filldraw[black] (8,1) circle (1.5pt);
      \filldraw[black] (7,2) circle (1.5pt);
      \filldraw[black] (6,3) circle (1.5pt);
      \draw (8,1) -- (6,3);
    \end{tikzpicture}
  \]
  Define partial transformations $O$ and $\Omega$ by (respectively) moving the base to lie parallel to the slope,
  provided this the result is a Ferrers diagram with unequal parts, and moving the slope to form a new base, provided
  that the result is a Ferrers diagram with unequal parts:
  \[
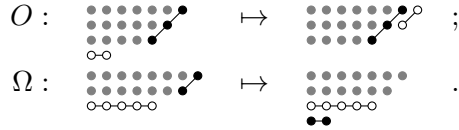
    \begin{aligned}
      O:\quad
      \begin{tikzpicture}[x=2mm,y=-2mm,baseline=-4mm]
        \foreach \y/\xsize in {1/6,2/5,3/4} {
          \foreach \x in {1,2,...,\xsize} {
            \filldraw[gray] (\x,\y) circle (1.5pt);
          }
        }
        \filldraw[black] (7,1) circle (1.5pt);
        \filldraw[black] (6,2) circle (1.5pt);
        \filldraw[black] (5,3) circle (1.5pt);
        \draw[black] (7,1) -- (5,3);
        \draw[black] (1,4) -- (2,4);
        \draw[black,fill=white] (1,4) circle (1.5pt);
        \draw[black,fill=white] (2,4) circle (1.5pt);
        \draw[white] (8,1) circle (1.5pt);
      \end{tikzpicture}
      \quad&\mapsto\quad
      \begin{tikzpicture}[x=2mm,y=-2mm,baseline=-4mm]
        \foreach \y/\xsize in {1/6,2/5,3/4} {
          \foreach \x in {1,2,...,\xsize} {
            \filldraw[gray] (\x,\y) circle (1.5pt);
          }
        }
        \filldraw[black] (7,1) circle (1.5pt);
        \filldraw[black] (6,2) circle (1.5pt);
        \filldraw[black] (5,3) circle (1.5pt);
        \draw[black] (7,1) -- (5,3);
        \draw[black] (8,1) -- (7,2);
        \draw[black,fill=white] (8,1) circle (1.5pt);
        \draw[black,fill=white] (7,2) circle (1.5pt);
      \end{tikzpicture}
      \quad; \\
      \Omega:\quad
      \begin{tikzpicture}[x=2mm,y=-2mm,baseline=-4mm]
        \foreach \y/\xsize in {1/7,2/6} {
          \foreach \x in {1,2,...,\xsize} {
            \filldraw[gray] (\x,\y) circle (1.5pt);
          }
        }
        \filldraw[black] (8,1) circle (1.5pt);
        \filldraw[black] (7,2) circle (1.5pt);
        \draw[black] (8,1) -- (7,2);
        \draw[black] (1,3) -- (5,3);
        \draw[black,fill=white] (1,3) circle (1.5pt);
        \draw[black,fill=white] (2,3) circle (1.5pt);
        \draw[black,fill=white] (3,3) circle (1.5pt);
        \draw[black,fill=white] (4,3) circle (1.5pt);
        \draw[black,fill=white] (5,3) circle (1.5pt);
      \end{tikzpicture}
      \quad&\mapsto\quad
      \begin{tikzpicture}[x=2mm,y=-2mm,baseline=-4mm]
        \foreach \y/\xsize in {1/7,2/6} {
          \foreach \x in {1,2,...,\xsize} {
            \filldraw[gray] (\x,\y) circle (1.5pt);
          }
        }
        \filldraw[black] (1,4) circle (1.5pt);
        \filldraw[black] (2,4) circle (1.5pt);
        \draw[black] (1,4) -- (2,4);
        \draw[black] (1,3) -- (5,3);
        \draw[black,fill=white] (1,3) circle (1.5pt);
        \draw[black,fill=white] (2,3) circle (1.5pt);
        \draw[black,fill=white] (3,3) circle (1.5pt);
        \draw[black,fill=white] (4,3) circle (1.5pt);
        \draw[black,fill=white] (5,3) circle (1.5pt);
        \filldraw[white] (8,1) circle (1.5pt);
      \end{tikzpicture}
      \quad.
    \end{aligned}
  \]
  Clearly $O$ and $\Omega$ are mutually inverse and so are partial bijections. The key to the proof is establishing that
  one of $O$ or $\Omega$ is defined except when the base and the slope meet and when the lengths of the base and slope
  are equal or when the length of the base exceeds the length of the slope by $1$:
  \begin{equation}
    \label{eq:oomegacases1}
    \begin{tikzpicture}[x=2mm,y=-2mm,baseline=-4mm]
      \foreach \y/\xsize in {1/6,2/5,3/4} {
        \foreach \x in {1,2,...,\xsize} {
          \filldraw[gray] (\x,\y) circle (1.5pt);
        }
      }
      \filldraw[black] (7,1) circle (1.5pt);
      \filldraw[black] (6,2) circle (1.5pt);
      \filldraw[black] (5,3) circle (1.5pt);
      \filldraw[black] (4,4) circle (1.5pt);
      \draw[black] (7,1) -- (4,4);
      \draw[black] (1,4) -- (4,4);
      \draw[black,fill=white] (1,4) circle (1.5pt);
      \draw[black,fill=white] (2,4) circle (1.5pt);
      \draw[black,fill=white] (3,4) circle (1.5pt);
      \clip (1,3.5) rectangle (4,4.5);
      \draw[black,fill=white] (4,4) circle (1.5pt);
    \end{tikzpicture}
    \qquad\qquad
    \begin{tikzpicture}[x=2mm,y=-2mm,baseline=-4mm]
      \foreach \y/\xsize in {1/5,2/4,3/3} {
        \foreach \x in {1,2,...,\xsize} {
          \filldraw[gray] (\x,\y) circle (1.5pt);
        }
      }
      \filldraw[black] (6,1) circle (1.5pt);
      \filldraw[black] (5,2) circle (1.5pt);
      \filldraw[black] (4,3) circle (1.5pt);
      \draw[black] (6,1) -- (4,3);
      \draw[black] (1,3) -- (4,3);
      \draw[black,fill=white] (1,3) circle (1.5pt);
      \draw[black,fill=white] (2,3) circle (1.5pt);
      \draw[black,fill=white] (3,3) circle (1.5pt);
      \clip (1,2.5) rectangle (4,3.5);
      \draw[black,fill=white] (4,3) circle (1.5pt);
    \end{tikzpicture}
  \end{equation}
  since it is precisely in these cases that for both $O$ and $\Omega$, attempting the transformations yields a
  configuration of dots that does not form a Ferrers diagram with unequal parts.

  The proof concludes by showing that $n = k(3k \pm 1)/2$, with an excess of one partition of odd or even type depending
  on the parity of $k$.
\end{proof}

Here, the key role played by spatial thinking is in realizing that $O$ or $\Omega$ is defined, except in the specified
cases. It is unnecessary to introduce notation or explicitly enumerate the cases for lengths of the base and slope. That
is, spatial reasoning allows \emph{simpler exhaustion of cases}. Further, it is spatial thinking that makes it possible
to use the condition \enquote{provided that the result is a Ferrers diagram with unequal parts}. Giving a non-spatial
definition of the partial transformations $O$ and $\Omega$ would require first defining the length of the base and the
slope of a partition $\lambda = \parens{\lambda_1,\ldots,\lambda_k}$ with unequal parts: the former is just $\lambda_k$,
but the length of the slope is the maximum $i$ such that $i \leq k$ and $\lambda_1,\ldots,\lambda_i$ is a sequence of
consecutive decreasing natural numbers. Then the map $O$ would be defined when $\lambda_k$ is less than this $i$. That
is, preliminary technical definitions would be required before the map $O$ can be defined. Thus we have an even stronger
instance of how spatial thinking can allow a simpler definition of a [partial] transformation, as discussed in
\fullref{Sections}{sec:ferrers} above. Here there a greater contrast between the simplicity of the spatial and non-spatial
definitions, even though the transformation itself is more complicated than \eqref{eq:conjugacyferrers}.

(In his original proof, Franklin \cite[pp.~448--9]{franklin_euler} does not attempt a general definition, but only
describes in text the transformation $O$ when the base contains $1$, $2$, or $3$ elements, and remarks that this can be
extended step by step to all partitions, excepting those \textquote[{\cite[p.~449 (my translation)]{franklin_euler}}]{in
  which the indicated process is not applicable}. Since Franklin does not give the conditions under which the
transformation $O$ is defined, even without proof, I suspect that his text is actually describing a process that he was
visualizing using Ferrers diagrams or some similar representation. Ferrers diagrams were first published by Sylvester in
1853 \cite{sylvester_impromptu}. Franklin's proof was published in 1881, so it is likely he was acquainted
with Ferrers diagrams. Furthermore, Sylvester \cite[p.~597]{sylvester_impromptu} speculates that Euler used used Ferrers
diagrams, so the idea may have been \enquote{in the air} for much longer. However, this is no more than a suspicion.)

\section{Spatial thinking, purity, and simplicity}

Arara \cite[pp.~209-211]{arana_onthealleged} has pointed out that, in many instances, impure proofs are held to have
greater simplicity. Impure proofs are those that establish a theorem by calling upon a different area of mathematics:
for example, Dirichlet's theorem that for any coprime natural numbers $a$ and $d$ there are infinitely many primes of
the form $a+nd$, a result whose statement is part of elementary number theory, was first proved by analysis. (For an
overview of the notion of purity in proof, see \cite{detlefsen_purity}.) In the examples concerning partitions above,
there certainly seems to impurity, although perhaps of a slightly different form: it is not another area of mathematics
that is called upon, but rather \enquote{pre-mathematical} spatial thinking.

Dedekind was partly motivated to develop the theory of what are now called Dedekind cuts to show that analytical
theorems could be re-expressed as theorems about natural numbers \cite[\S~V]{stein_logos}, which would alleviate the
epistemological concerns of impurity in proof (which go back to Aristotle's \enquote{met\'{a}basis eis \'{a}llo
  g\'{e}nos}). Similarly, as discussed above, the proofs using Ferrers diagrams could be re-expressed in non-spatial
terms. But, as shown, this change would immensely decrease the simplicity. Ferrers diagrams are more than simply another
form of notation like a sequence $\parens{\lambda_1,\ldots,\lambda_k}$. They are more than just a book-keeping
device. They are a vehicle that enables spatial thinking about partitions, and it is this spatial thinking that yields
the simpler proofs.

\bibliography{\jobname}
\bibliographystyle{alphaabbrv}

\end{document}